%% file: main.tex
\documentclass[11pt]{article}
\usepackage[textwidth=160mm,textheight=237mm]{geometry}
\usepackage[utf8]{inputenc}
\usepackage{amsmath,amssymb,amsthm,mathtools}
\usepackage{cleveref}
\usepackage{tikz}
\usepackage{url}
\usepackage{multicol}
\usepackage[width=136mm]{caption}
\usepackage{enumerate}

\newtheorem{theorem}{Theorem}

\newtheorem{definition}[theorem]{Definition}

\DeclareMathOperator{\tr}{tr}
\DeclareMathOperator{\sgn}{sgn}

\DeclareMathOperator{\lspan}{span}

\title{Limit cycles in mass-conserving deficiency-one\\ mass-action systems\footnote{BB was supported by the Austrian Science Fund (FWF), project P32532.}}
\author{Bal\'azs Boros and Josef Hofbauer}
\date{\small{University of Vienna, Austria}}

\begin{document}

\maketitle

\begin{abstract}
We present some simple mass-action systems with limit cycles that fall under the scope of the Deficiency-One Theorem. All the constructed examples are mass-conserving and their stoichiometric subspace is two-dimensional. Using the continuation software MATCONT, we depict the limit cycles in all stoichiometric classes at once. The networks are trimolecular and tetramolecular, and some exhibit two or even three limit cycles. Finally, we show that the associated mass-action system of a bimolecular reaction network with two-dimensional stoichiometric subspace does not admit a limit cycle.


\end{abstract}

\input{sections/1_intro}
\input{sections/2_mass_action_systems}
\input{sections/3_parallelograms}
\input{sections/4_bimolecular}

\bibliographystyle{abbrv}
\bibliography{biblio}

\end{document}

%% file: sections/1_intro.tex
\section{Introduction}

Recently we have constructed a number of planar deficiency-one mass-action systems that oscillate \cite{boros:hofbauer:2021a}. In this paper we provide a couple of examples that admit limit cycles, all with three species (whose concentrations are denoted by $x$, $y$, $z$) and a linear conservation law $d_1 \dot{x} + d_2 \dot{y} + d_3 \dot{z} = 0$ with $d_1$, $d_2$, $d_3 >0$. Thus, $d_1 x + d_2 y + d_3 z = c$ for some $c>0$ holds for all time and hence, beyond the rate constants, there is another parameter. The main approach we follow is that we regard this additional parameter, $c$, a bifurcation parameter. We will be interested in the stability of equilibria and limit cycles as $c$ varies, and in the case of limit cycles their number can also vary by $c$. We will see examples with multiple limit cycles, in one of the cases we can even prove the existence of three limit cycles that are all born via a degenerate Andronov--Hopf bifurcation. Probably the most interesting phenomenon found in this paper is a mass-action system with the unique positive equilibrium being asymptotically stable for all $c>0$, but not globally stable for $c_1^* \leq c \leq c^*_2$, because a torus formed of stable and unstable limit cycles surrounds the curve of equilibria. We extensively use the continuation software MATCONT \cite{dhooge:govaerts:kuznetsov:2003} to visualise the limit cycles in the $(x,y,z)$-space for all $c>0$ at once, while fixing all the rate constants.

A common feature of the networks analysed in this paper is that they all have at least one chemical complex that is trimolecular or tetramolecular. This is necessary, as the mass-action system associated to a bimolecular reaction network with two-dimensional stoichiometric subspace does not admit limit cycles. This latter fact is known for $2$ and $3$ species \cite{pota:1983,pota:1985}, while for arbitrary number of species we prove it in \Cref{sec:bimolecular}. In fact, we show that essentially the only bimolecular reaction networks with a two-dimensional stoichiometric subspace whose associated mass-action system oscillates are the Lotka and the Ivanova reactions, where each positive non-equilibrium solution is periodic.

The rest of this paper is organised as follows. In \Cref{sec:mass_action} we collect the needed terminology and some basic results from chemical reaction network theory. In \Cref{sec:parallelograms} we present a number of mass-conserving networks that admit (multiple) limit cycles. Finally, in \Cref{sec:bimolecular} we show that the associated mass-action system of a bimolecular reaction network with two-dimensional stoichiometric subspace does not admit a limit cycle.

%% file: sections/2_mass_action_systems.tex
\section{Mass-action systems}
\label{sec:mass_action}

In this section we briefly introduce mass-action systems and related notions that are necessary for our exposition. For more details about mass-action systems, consult e.g. \cite{feinberg:1987}, \cite{gunawardena:2003}. The symbols $\mathbb{R}_+$ and $\mathbb{Z}_{\geq0}$ denote the set of positive real numbers and the set of nonnegative integers, respectively.

\begin{definition}
A {\emph{Euclidean embedded graph}} (or a \emph{reaction network}) is a directed graph $(V,E)$, where $V$ is a nonempty finite subset of $\mathbb{Z}^n_{\geq0}$.
\end{definition}

Denote by $\mathsf{X}_1, \ldots, \mathsf{X}_n$ the $n$ \emph{species} and by $y^{1}, \ldots, y^{m}$ the elements of $V$, called \emph{complexes}. Accordingly, we often refer to $y^{i}$ as $y^{i}_1\mathsf{X}_1 + \cdots + y^{i}_n\mathsf{X}_n$. The entries of $y^{i}$ are the \emph{stoichiometric coefficients}. The concentrations of the species $\mathsf{X}_1, \ldots, \mathsf{X}_n$ at time $\tau$ are collected in the vector $x(\tau) \in \mathbb{R}^n_+$.

\begin{definition}
A {\emph{mass-action system}} is a triple $(V,E, \kappa)$, where $(V,E)$ is a reaction network and $\kappa\colon E \to \mathbb{R}_+$ is the collection of the \emph{rate constants}. Its \emph{associated differential equation} on $\mathbb{R}^n_+$ is
\begin{align}
\label{eq:mass_action_ode}
\dot{x}(\tau) = \sum_{(i,j)\in E} \kappa_{ij} x_1(\tau)^{y^{i}_1}\cdots x_n(\tau)^{y^{i}_n} (y^{j}-y^{i}).
\end{align}
\end{definition}

The span $\mathcal{S} = \lspan\{y^j-y^i\colon (i,j)\in E\}\leq \mathbb{R}^n$ is called the \emph{stoichiometric subspace} and the sets $(p+\mathcal{S})\cap\mathbb{R}^n_+$ for $p\in\mathbb{R}^n_+$ are called the \emph{(positive) stoichiometric classes}. The stoichiometric classes provide a foliation of the positive orthant $\mathbb{R}^n_+$ into forward invariant sets of the mass-action differential equation \eqref{eq:mass_action_ode}. Therefore, dynamical questions (e.g. existence, uniqueness, stability, or number of equilibria or limit cycles) are examined relative to a stoichiometric class. The \emph{rank} of a reaction network (or its associated mass-action system) is defined to be the dimension of its stoichiometric subspace.

In some cases, a network property alone has consequences on the qualitative behaviour of the differential equation \eqref{eq:mass_action_ode}. For instance, if the directed graph $(V,E)$ is strongly connected (i.e., for all $i,j\in V$ there exists a directed path from $i$ to $j$) then the associated mass-action differential equation is permanent \cite[Theorem 1.3]{gopalkrishnan:miller:shiu:2014}, \cite[Theorem 5.5]{anderson:cappelletti:kim:nguyen:2020}, \cite[Theorem 4.2]{boros:hofbauer:2020}. We now define permanence.

\begin{definition}
A mass-action system is \emph{permanent} in a stoichiometric class $\mathcal{P}$ if there exists a compact set $K\subseteq\mathcal{P}$ with the property that for each solution $\tau \mapsto x(\tau)$ with $x(0)\in\mathcal{P}$ there exists a $\tau_0\geq0$ such that $x(\tau)\in K$ holds for all $\tau\geq\tau_0$. A mass-action system is permanent if it is permanent in every stoichiometric class.
\end{definition}

\begin{theorem}[\cite{gopalkrishnan:miller:shiu:2014,anderson:cappelletti:kim:nguyen:2020,boros:hofbauer:2020}]
\label{thm:wr_permanence}
If $(V,E)$ is strongly connected then the mass-action system $(V,E,\kappa)$ is permanent.
\end{theorem}

We now recall a classical theorem on the number of positive equilibria for mass-action systems with low deficiency. The \emph{deficiency} of a reaction network $(V,E)$ is the nonnegative integer $\delta = m - \ell - \dim\mathcal{S}$, where $m = \lvert V \rvert$, $\ell$ is the number of connected components of the directed graph $(V,E)$, and $\mathcal{S}$ is the stoichiometric subspace.

\begin{theorem}[Deficiency-One Theorem \cite{feinberg:1995}]
\label{thm:dfc1}
Assume that the reaction network $(V,E)$ is strongly connected and its deficiency is zero or one. Then the following statements hold.
\begin{enumerate}[(i)]
\item There exists a unique positive equilibrium in every stoichiometric class.
\item The set of positive equilibria equals $\{x \in \mathbb{R}^n_+ \colon \log x - \log x^* \in \mathcal{S}^\perp\}$, where $x^*$ is any given positive equilibrium.
\item Denoting by $J\in\mathbb{R}^{n\times n}$ the Jacobian matrix at a positive equilibrium,
\begin{enumerate}[(a)]
    \item the linear map $J|_\mathcal{S}\colon \mathcal{S} \to \mathcal{S}$ is nonsingular and
    \item the linear map $-J|_\mathcal{S}\colon \mathcal{S} \to \mathcal{S}$ is orientation-preserving, i.e., its determinant is positive, or equivalently $\sgn \det J|_\mathcal{S} = (-1)^{\dim \mathcal{S}}$.
\end{enumerate}
\end{enumerate}
\end{theorem}
\begin{proof}
Parts (i), (ii), (iii)(a) are proven in \cite{feinberg:1995}.

We now prove part (iii)(b). For a fixed stoichiometric class, consider the map $\mathbb{R}^E_+ \to \mathbb{R}$ that assigns to each set of rate constants the determinant of the restricted Jacobian map at the unique positive equilibrium. This map is continuous, it is everywhere nonzero by part (iii)(a), and hence has a constant sign. For the subset of complex balanced systems the equilibrium is linearly stable (see \cite[Theorem 4.3.2]{johnston:2011}, \cite[Theorem 15.2.2]{feinberg:2019}, or \cite[Theorem 8]{boros:mueller:regensburger:2020}), and hence the sign is $(-1)^{\dim \mathcal{S}}$.
\end{proof}
As a consequence of part (iii)(b) in \Cref{thm:dfc1}, when $\dim\mathcal{S} = 2$ then the product of the two nonzero eigenvalues at a positive equilibrium is positive, hence it is enough to look at the trace for deciding stability: if the trace is negative (respectively, positive) then the equilibrium is asymptotically stable (respectively, repelling) within its stoichiometric class. When the trace vanishes, the two nonzero eigenvalues are purely imaginary and the stability can be decided by computing the focal values.

In the special case when $\dim \mathcal{S} = n-1$, part (ii) in \Cref{thm:dfc1} has the immediate consequence that the set of positive equilibria can be parametrised as follows:
\begin{align}\label{eq:toric}
\{(x^*_1 t^{d_1},\ldots,x^*_n t^{d_n}) \colon t>0\},
\end{align}
where $x^* \in\mathbb{R}^n_+$ is any positive equilibrium and $d$ is any nonzero vector in $\mathcal{S}^\perp$.

A reaction network is \emph{mass-conserving} if the stoichiometric classes are bounded, or equivalently there exists a $d \in \mathcal{S}^\perp$ with all coordinates being positive.

In \Cref{subsec:1LC,subsec:2LC_1Hopf,subsec:2LC_2Hopf}, all the networks fall under the scope of the \Cref{thm:dfc1} and satisfy $n = 3$ and $\dim \mathcal{S} = 2$. Thus, in all of the examples the set of positive equilibria is parametrised as in \eqref{eq:toric} and their stability is decided by the sign of the trace of the Jacobian matrix (provided it is nonzero).

%% file: sections/3_parallelograms.tex
\section{Parallelograms}
\label{sec:parallelograms}

Our goal is to find some simple mass-conserving reaction networks that fall under the scope of the Deficiency-One Theorem and their associated mass-action systems admit limit cycles. By the recent result \cite[Theorem 1]{banaji:boros:hofbauer:2021}, a way to achieve this is the following: first find a planar mass-action system with a limit cycle (or multiple limit cycles) and then add a new species to some of the reactions in such a way that the rank of the new network is still two and the stoichiometric classes become bounded. For instance, one confirms that all the three parallelograms in \Cref{fig:planar_parallelograms} admit limit cycles. This can be proven by showing that there exist rate constants such that the unique positive equilibrium is repelling. The existence of a stable limit cycle then follows from the permanence of the system and the Poincar\'e--Bendixson Theorem. One finds that each of the second and the third parallelograms admits even two limit cycles (to show this, one calculates the first focal value at the positive equilibrium and finds that it can be positive for some rate constants, allowing an unstable limit cycle to be born via a subcritical Andronov--Hopf bifurcation).

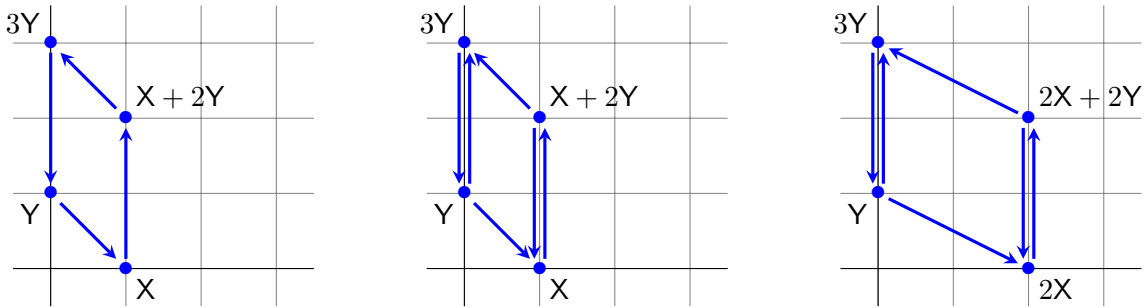
\begin{figure}[h]
\centering
\input{tikz/parallelogram_planar}
\caption{Three planar reaction networks that all fall under the scope of the Deficiency-One Theorem and their associated mass-action systems admit limit cycles. Each of the second and the third parallelograms admits even two limit cycles: there exist rate constants such that the unique positive equilibrium is asymptotically stable and is surrounded by an unstable and a stable limit cycle.}
\label{fig:planar_parallelograms}
\end{figure}

In \Cref{subsec:1LC,subsec:2LC_1Hopf,subsec:2LC_2Hopf}, we respectively lift the three parallelograms in \Cref{fig:planar_parallelograms} by adding a new species. In all the three cases, we obtain a three species mass-action system that is mass-conserving and that falls under the scope of the Deficiency-One Theorem. Thus, the set of positive equilibria is of the form \eqref{eq:toric}, it intersects every stoichiometric class in exactly one point. The aim is to find rate constants such that the positive equilibria are surrounded by limit cycles. In fact, we adopt the approach that the rate constants are fixed, and regard the stoichiometric class as a bifurcation parameter. In a number of cases, we visualise the (stable and unstable) limit cycles across the stoichiometric classes. This is performed using the numerical continuation software MATCONT \cite{dhooge:govaerts:kuznetsov:2003}.

\input{sections/3p1_1LC}
\input{sections/3p2_2LC}
\input{sections/3p3_2Hopf}

%% file: tikz/parallelogram_planar.tex
\begin{tikzpicture}

\draw [step=1, gray, very thin] (-0.5,-0.5) grid (3.5,3.5);
\draw [ -, black] (-0.5,0)--(3.5,0);
\draw [ -, black] (0,-0.5)--(0,3.5);

\node[inner sep=0,outer sep=1] (P1) at (0,1) {\large \textcolor{blue}{$\bullet$}};
\node[inner sep=0,outer sep=1] (P2) at (1,0) {\large \textcolor{blue}{$\bullet$}};
\node[inner sep=0,outer sep=1] (P3) at (1,2) {\large \textcolor{blue}{$\bullet$}};
\node[inner sep=0,outer sep=1] (P4) at (0,3) {\large \textcolor{blue}{$\bullet$}};

\node [below left]  at (P1) {$\mathsf{Y}$};
\node [below right] at (P2) {$\mathsf{X}$};
\node [above right] at (P3) {$\mathsf{X}+2\mathsf{Y}$};
\node [above left]  at (P4) {$3\mathsf{Y}$};

\draw[arrows={-stealth},very thick,blue] (P1) to node {} (P2);
\draw[arrows={-stealth},very thick,blue] (P2) to node {} (P3);
\draw[arrows={-stealth},very thick,blue] (P3) to node {} (P4);
\draw[arrows={-stealth},very thick,blue] (P4) to node {} (P1);

\begin{scope}[shift={(5.5,0)}]

\draw [step=1, gray, very thin] (-0.5,-0.5) grid (3.5,3.5);
\draw [ -, black] (-0.5,0)--(3.5,0);
\draw [ -, black] (0,-0.5)--(0,3.5);

\node[inner sep=0,outer sep=1] (P1) at (0,1) {\large \textcolor{blue}{$\bullet$}};
\node[inner sep=0,outer sep=1] (P2) at (1,0) {\large \textcolor{blue}{$\bullet$}};
\node[inner sep=0,outer sep=1] (P3) at (1,2) {\large \textcolor{blue}{$\bullet$}};
\node[inner sep=0,outer sep=1] (P4) at (0,3) {\large \textcolor{blue}{$\bullet$}};

\node [below left]  at (P1) {$\mathsf{Y}$};
\node [below right] at (P2) {$\mathsf{X}$};
\node [above right] at (P3) {$\mathsf{X}+2\mathsf{Y}$};
\node [above left]  at (P4) {$3\mathsf{Y}$};

\draw[arrows={-stealth},very thick,blue] (P1) to node {} (P2);
\draw[arrows={-stealth},very thick,blue,transform canvas={xshift=2pt}] (P2) to node {} (P3);
\draw[arrows={-stealth},very thick,blue,transform canvas={xshift=-2pt}] (P3) to node {} (P2);
\draw[arrows={-stealth},very thick,blue] (P3) to node {} (P4);
\draw[arrows={-stealth},very thick,blue,transform canvas={xshift=2pt}] (P1) to node {} (P4);
\draw[arrows={-stealth},very thick,blue,transform canvas={xshift=-2pt}] (P4) to node {} (P1);

\end{scope}

\begin{scope}[shift={(11,0)}]

\draw [step=1, gray, very thin] (-0.5,-0.5) grid (3.5,3.5);
\draw [ -, black] (-0.5,0)--(3.5,0);
\draw [ -, black] (0,-0.5)--(0,3.5);

\node[inner sep=0,outer sep=1] (P1) at (0,1) {\large \textcolor{blue}{$\bullet$}};
\node[inner sep=0,outer sep=1] (P2) at (2,0) {\large \textcolor{blue}{$\bullet$}};
\node[inner sep=0,outer sep=1] (P3) at (2,2) {\large \textcolor{blue}{$\bullet$}};
\node[inner sep=0,outer sep=1] (P4) at (0,3) {\large \textcolor{blue}{$\bullet$}};

\node [below left]  at (P1) {$\mathsf{Y}$};
\node [below right] at (P2) {$2\mathsf{X}$};
\node [above right] at (P3) {$2\mathsf{X}+2\mathsf{Y}$};
\node [above left]  at (P4) {$3\mathsf{Y}$};

\draw[arrows={-stealth},very thick,blue] (P1) to node {} (P2);
\draw[arrows={-stealth},very thick,blue,transform canvas={xshift=2pt}] (P2) to node {} (P3);
\draw[arrows={-stealth},very thick,blue,transform canvas={xshift=-2pt}] (P3) to node {} (P2);
\draw[arrows={-stealth},very thick,blue] (P3) to node {} (P4);
\draw[arrows={-stealth},very thick,blue,transform canvas={xshift=2pt}] (P1) to node {} (P4);
\draw[arrows={-stealth},very thick,blue,transform canvas={xshift=-2pt}] (P4) to node {} (P1);

\end{scope}

\end{tikzpicture}

%% file: sections/3p1_1LC.tex
\subsection{Supercritical Andronov--Hopf bifurcation}
\label{subsec:1LC}

Let us take the first planar parallelogram in \Cref{fig:planar_parallelograms} and add a new species, $\mathsf{Z}$, with stoichiometric coefficient $\gamma>0$ as follows:
\begin{align}\label{eq:ode_1LC}
\begin{aligned}
\input{tikz/parallelogram_1LC}
\end{aligned}
\end{align}
where we also displayed the associated mass-action differential equation. Note that $\gamma x + \gamma y + 2 z$ is a conserved quantity, and therefore the network is mass-conserving. By a short calculation, the set of positive equilibria is the curve
\begin{align*}
\left(\left(\frac{\kappa_1\kappa_4}{\kappa_2\kappa_3}\right)^\frac12 t^\gamma,
t^\gamma,
\left(\frac{\kappa_3\kappa_4}{\kappa_1\kappa_2}\right)^\frac{1}{2\gamma}t^2\right) \text{ for } t>0.
\end{align*}
Since the stability of a positive equilibrium within its stoichiometric class is determined by the sign of the trace of the Jacobian matrix, we compute the trace along the curve of equilibria and find it is
\begin{align} \label{eq:1LC_trace}
\left(\left(\frac{\kappa_1\kappa_3\kappa_4}{\kappa_2}\right)^\frac12-(\kappa_3+6\kappa_4)\right)t^{2\gamma}-\gamma^2\kappa_4\left(\frac{\kappa_1\kappa_2}{\kappa_3\kappa_4}\right)^{\frac{1}{2\gamma}}t^{3\gamma-2}  \text{ for } t>0.
\end{align}
As a function of $t$, it behaves differently for $\gamma<2$, $\gamma=2$, and $\gamma>2$. In \Cref{subsec:1LC_homog,subsec:1LC_nonhomog} we study the cases $\gamma = 2$ and $\gamma \neq 2$, respectively.

\subsubsection{Case $\gamma=2$}
\label{subsec:1LC_homog}

Notice that for $\gamma=2$, every complex in the network is trimolecular. As a consequence, the r.h.s.\ of the mass-action differential equation is a homogeneous polynomial (every monomial is of degree three). Furthermore, $x+y+z$ is conserved and it is not hard to see that in every stoichiometric class the dynamics is the same (up to scaling).

For $\gamma=2$, the set of positive equilibria is a half-line and the trace of the Jacobian matrix along that, formula \eqref{eq:1LC_trace}, equals
\begin{align*}
\left(\left(\frac{\kappa_1\kappa_3\kappa_4}{\kappa_2}\right)^\frac12-\kappa_3-6\kappa_4- 4\kappa_4\left(\frac{\kappa_1\kappa_2}{\kappa_3\kappa_4}\right)^{\frac{1}{4}}\right)t^4,
\end{align*}
an expression whose sign is independent of $t$. With $a(\kappa)$ denoting the coefficient of $t^4$, every positive equilibrium is asymptotically stable (respectively, repelling) if $a(\kappa)<0$ (respectively, $a(\kappa)>0$). For $a(\kappa)=0$ one computes the first focal value and finds it is negative, implying that the equilibria are asymptotically stable and the corresponding Andronov--Hopf bifurcation is supercritical. Thus, setting the rate constants such that $a(\kappa)=0$ and then perturbing them slightly to achieve $a(\kappa)>0$, results in the emergence of a stable limit cycle in every stoichiometric class. By the homogeneity, the phase portrait is the same in every stoichiometric class, and the limit cycles that are born via a supercritical Andronov--Hopf bifurcation indeed coexist in every stoichiometric class. In fact, since the system is permanent and the equilibrium is repelling for $a(\kappa)>0$, a stable limit cycle exists in every stoichiometric class for all rate constants with $a(\kappa)>0$. We depicted these limit cycles in \Cref{fig:1LC_homog} with $\kappa_1 = 16$, $\kappa_2 = \frac{1}{16}$, $\kappa_3 = 1$, $\kappa_4 = 1$ (thus, $a(\kappa)=5>0$).

\subsubsection{Case $\gamma\neq2$}
\label{subsec:1LC_nonhomog}

When $\gamma\neq2$, the exponents $2\gamma$ and $3\gamma-2$ of $t$ in the trace formula \eqref{eq:1LC_trace} are unequal. Since the coefficient of $t^{3\gamma-2}$ is negative, the trace vanishes at some $t>0$ if and only if the coefficient of $t^{2\gamma}$ is positive. Thus, if $\sqrt{\frac{\kappa_1}{\kappa_2}} \leq \sqrt{\frac{\kappa_3}{\kappa_4}} + 6\sqrt{\frac{\kappa_4}{\kappa_3}}$ then every positive equilibrium is asymptotically stable. On the other hand, if $\sqrt{\frac{\kappa_1}{\kappa_2}} > \sqrt{\frac{\kappa_3}{\kappa_4}} + 6\sqrt{\frac{\kappa_4}{\kappa_3}}$ then the trace vanishes at exactly one positive $t$, call it $t^*$. One computes the first focal value at $t=t^*$ and finds it is negative (due to the high computational complexity, we could, unfortunately, verify this only for some specific values of $\gamma$, e.g. $1$, $3$, $4$). Therefore, the equilibrium at $t=t^*$ is asymptotically stable and the corresponding Andronov--Hopf bifurcation is supercritical (regarding $t$ as a parameter, while all the rate constants are fixed). Hence, we obtain the following qualitative pictures:

\medskip

\textbf{Case $\gamma<2$.} For $t\leq t^*$ the equilibrium is asymptotically stable, for $t>t^*$ the equilibrium is repelling, and for $t$ slightly larger than $t^*$ there exists a stable limit cycle that is born via a supercritical Andronov--Hopf bifurcation. In fact, since the system is permanent and the equilibrium is repelling for all $t>t^*$, a stable limit cycle exists for all $t>t^*$. Using MATCONT, we depicted these limit cycles in the left panel of \Cref{fig:1LC_nonhomog} with $\gamma = 1$, $\kappa_1 = 8$, $\kappa_2 = \frac{1}{8}$, $\kappa_3 = 1$, $\kappa_4 = 1$.

\medskip

\textbf{Case $\gamma>2$.} For $t\geq t^*$ the equilibrium is asymptotically stable, for $t<t^*$ the equilibrium is repelling, and for $t$ slightly smaller than $t^*$ there exists a stable limit cycle that is born via a supercritical Andronov--Hopf bifurcation. In fact, since the system is permanent and the equilibrium is repelling for all $t<t^*$, a stable limit cycle exists for all $t<t^*$. Using MATCONT, we depicted these limit cycles in the right panel of \Cref{fig:1LC_nonhomog} with $\gamma = 3$, $\kappa_1 = 16$, $\kappa_2 = \frac{1}{16}$, $\kappa_3 = 1$, $\kappa_4 = 1$.

\begin{figure}[ht]
    \centering
    \includegraphics[scale=0.229]{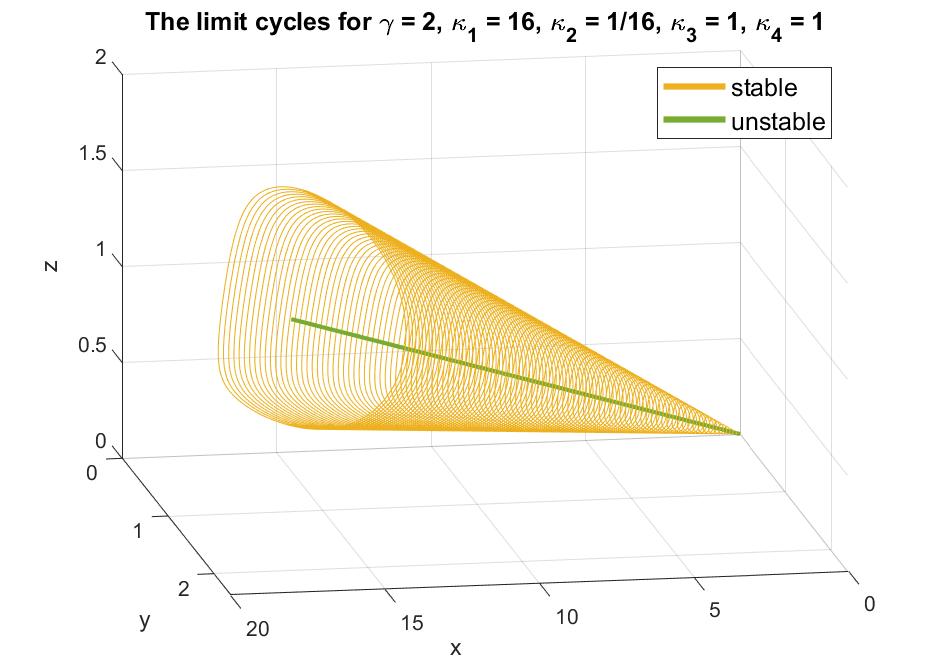}
    \caption{The stable limit cycles that are born via a supercritical Andronov--Hopf bifurcation for the mass-action system \eqref{eq:ode_1LC} (case $\gamma=2$).}
    \label{fig:1LC_homog}
\end{figure}

\begin{figure}[ht]
    \centering
    \begin{tabular}{cc}
    \includegraphics[scale=0.229]{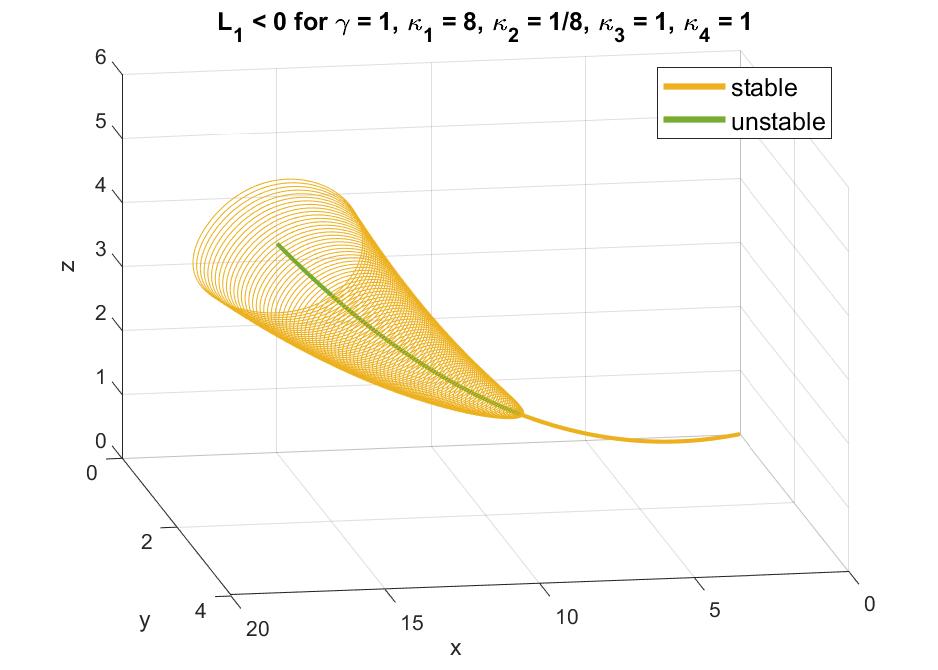} &
    \includegraphics[scale=0.229]{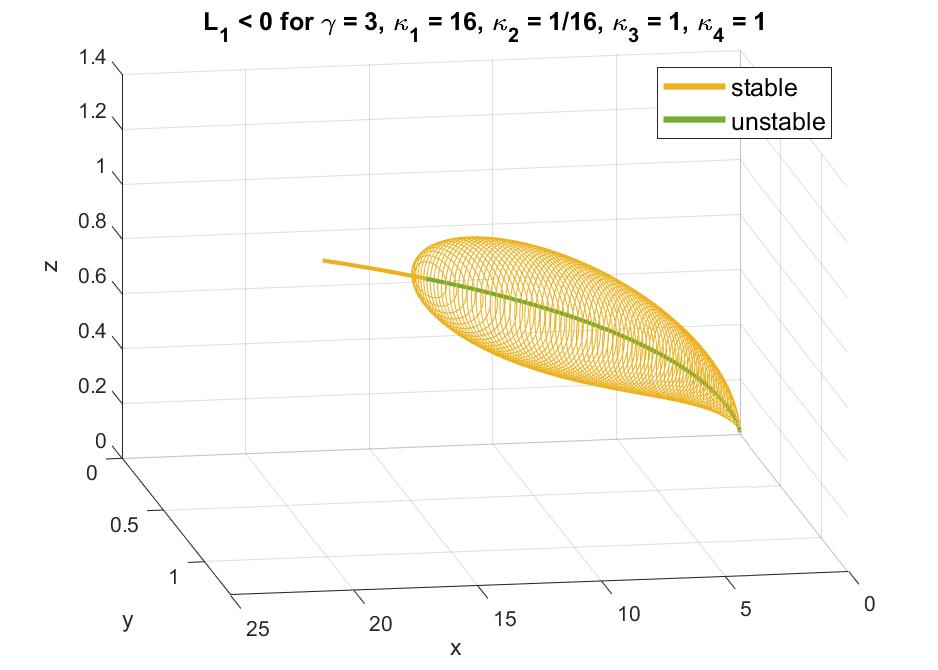}
    \end{tabular}
    \caption{The stable limit cycles that are born via a supercritical Andronov--Hopf bifurcation for the mass-action system \eqref{eq:ode_1LC} (case $\gamma<2$ on the left, case $\gamma>2$ on the right).}
    \label{fig:1LC_nonhomog}
\end{figure}

%% file: tikz/parallelogram_1LC.tex
\begin{tikzpicture}[scale=1.4]

\node (P1) at (0,1)  {$\mathsf{Y}+\gamma \mathsf{Z}$};
\node (P2) at (1,0)  {$\mathsf{X}+\gamma \mathsf{Z}$};
\node (P3) at (1,2)  {$\mathsf{X}+2\mathsf{Y}$};
\node (P4) at (0,3)  {$3\mathsf{Y}$};

\draw[arrows={-stealth},thick] (P1) to node[below left] {$\kappa_1$} (P2);
\draw[arrows={-stealth},thick] (P2) to node[right] {$\kappa_2$} (P3);
\draw[arrows={-stealth},thick] (P3) to node[above right] {$\kappa_3$} (P4);
\draw[arrows={-stealth},thick] (P4) to node[left] {$\kappa_4$} (P1);

\node at (5,1.5) {$\begin{aligned}
\dot{x} &= \kappa_1 y z^\gamma - \kappa_3 x y^2, \\
\dot{y} &= -\kappa_1 y z^\gamma + \kappa_3 x y^2 + 2(\kappa_2 x z^\gamma - \kappa_4 y^3), \\
\dot{z} &= \gamma(-\kappa_2 x z^\gamma + \kappa_4 y^3),
\end{aligned}$};

\end{tikzpicture}

%% file: sections/3p2_2LC.tex
\subsection{Subcritical Andronov--Hopf bifurcation}
\label{subsec:2LC_1Hopf}

Let us take the second planar parallelogram in \Cref{fig:planar_parallelograms} and add a new species, $\mathsf{Z}$, with stoichiometric coefficient $\gamma>0$ as follows:
\begin{align}\label{eq:ode_2LC}
\begin{aligned}
\input{tikz/parallelogram_2LC}
\end{aligned}
\end{align}
where we employed special rate constants (for $a,b>0$) and also displayed the associated mass-action differential equation. This special choice of the rate constants makes the calculations somewhat easier and they still allow us to present some qualitative pictures that were not seen in \Cref{subsec:1LC}. Note that $\gamma x + \gamma y + 2 z$ is a conserved quantity, and therefore the network is mass-conserving. By a short calculation, the set of positive equilibria is the curve
\begin{align*}
\left(2 t^\gamma,t^\gamma/2,t^2\right) \text{ for } t>0.
\end{align*}
Since the stability of a positive equilibrium within its stoichiometric class is determined by the sign of the trace of the Jacobian matrix, we compute the trace along the curve of equilibria and find it is
\begin{align} \label{eq:2LC_trace}
4t^{2\gamma}\left[\left(\frac{3}{16}-4a-b\right)-\frac{\gamma^2}{8}(4a+b)t^{\gamma-2}\right] \text{ for } t>0.
\end{align}
As a function of $t$, it behaves differently for $\gamma<2$, $\gamma=2$, and $\gamma>2$. In \Cref{subsec:2LC_homog,subsec:2LC_nonhomog} we study the cases $\gamma = 2$ and $\gamma \neq 2$, respectively.

\subsubsection{Case $\gamma=2$}
\label{subsec:2LC_homog}

First notice that for $\gamma=2$, the same way as in \Cref{subsec:1LC_homog}, the phase portrait is the same (up to scaling) in every stoichiometric class.

For $\gamma=2$, the set of positive equilibria is a half-line and the trace of the Jacobian matrix along that, formula \eqref{eq:2LC_trace}, equals $6 t^4\left(\frac18-4a-b\right)$. Thus, all positive equilibria are asymptotically stable if $4a+b>\frac18$, while all of them are repelling if $4a+b<\frac18$. On the $4a+b=\frac18$ line in parameter space, one computes the first focal value and gets
\begin{align*}
L_1 =  1280b^2-16b-7 \text{ for } 0 < b < \frac{1}{8}.
\end{align*}
Hence, with $b^* = \frac{1+\sqrt{141}}{160}\approx 0.08$, the first focal value is negative for $0<b<b^*$, vanishes at $b^*$, and is positive for $b^*<b<\frac{1}{8}$. This allows us to construct two limit cycles in each stoichiometric class in the following way. Take $a=\frac{1}{200}$ and $b=\frac{21}{200}$. Then the trace vanishes and the first focal value is positive, so the equilibrium is repelling. By permanence, there exists a stable limit cycle. By increasing $b$ a tiny bit, the trace becomes negative, and an unstable limit cycle is born via a subcritical Andronov--Hopf bifurcation. Thus, there exist $a$ and $b$ such that two limit cycles coexist. A numerical experiment suggests that these two limit cycles merge and disappear through a fold bifurcation around $b\approx\frac{23.3}{200}$. Using MATCONT, we depicted in \Cref{fig:2LC_homog} the two nested cones of limit cycles for $a = \frac{1}{200}$ and $b = \frac{22}{200}$.

\begin{figure}[ht]
    \centering
    \includegraphics[scale=0.229]{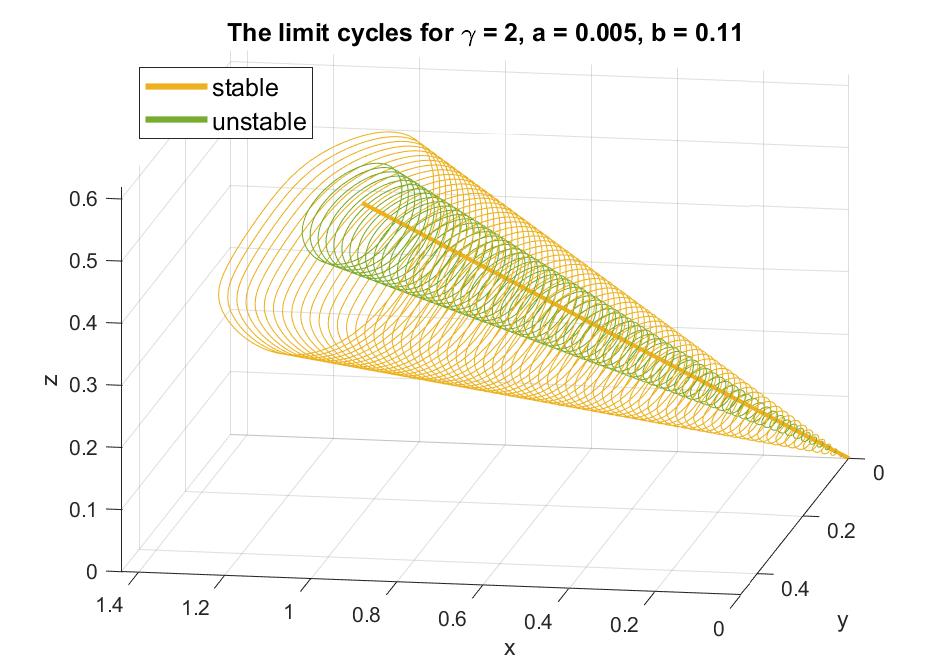}
    \caption{The two nested cones of limit cycles, the unstable limit cycles are born via a subcritical Andronov--Hopf bifurcation for the mass-action system \eqref{eq:ode_2LC} (case $\gamma=2$).}
    \label{fig:2LC_homog}
\end{figure}

\subsubsection{Case $\gamma \neq 2$}
\label{subsec:2LC_nonhomog}

When $\gamma\neq2$, for any fixed $a,b>0$ with $4a+b<\frac{3}{16}$ there exists a unique $t^*>0$ for which the trace \eqref{eq:2LC_trace} vanishes. To decide stability of the equilibrium, one computes the first focal value and finds it can have any sign, see the left (case $\gamma=1$) and right (case $\gamma=3$) panels in the top row in \Cref{fig:2LC_nonhomog}. In particular, the first focal value can be positive and the corresponding Andronov--Hopf bifurcation is then subcritical, allowing an unstable limit cycle to be born. Since the system is permanent, each unstable equilibrium and each unstable limit cycle is surrounded by a stable limit cycle. Using MATCONT, we depicted these limit cycles in the bottom row of \Cref{fig:2LC_nonhomog} for some particular choices of the parameters.

\begin{figure}[ht]
    \centering
    \begin{tabular}{cc}
    \includegraphics[scale=0.5]{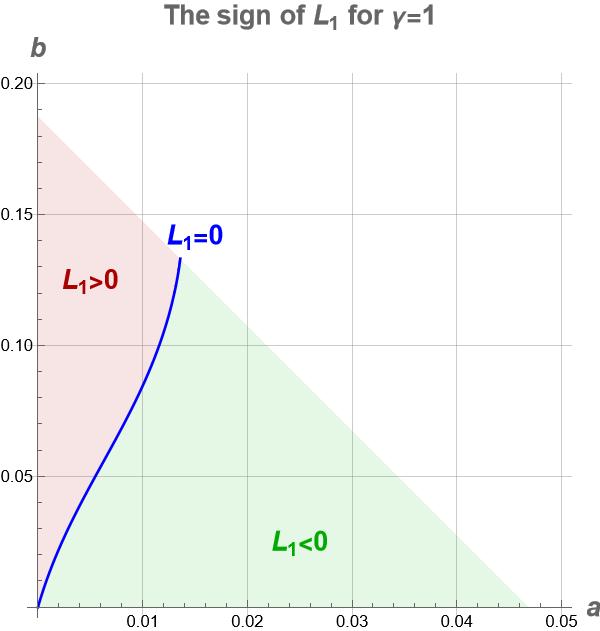} &
    \includegraphics[scale=0.5]{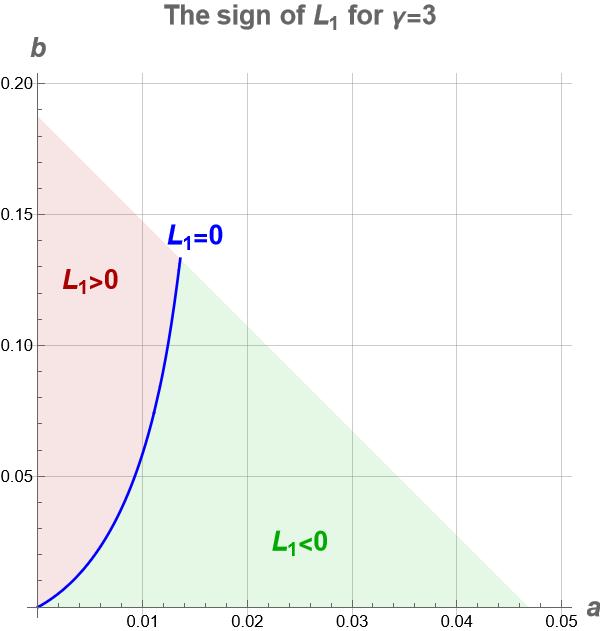} \\
    \includegraphics[scale=0.229]{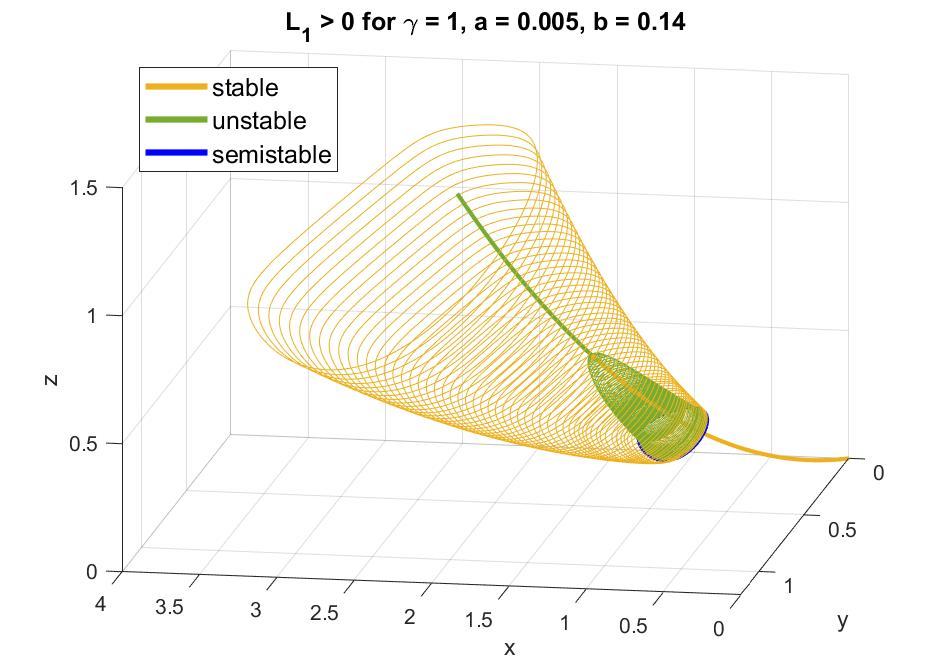} &
    \includegraphics[scale=0.229]{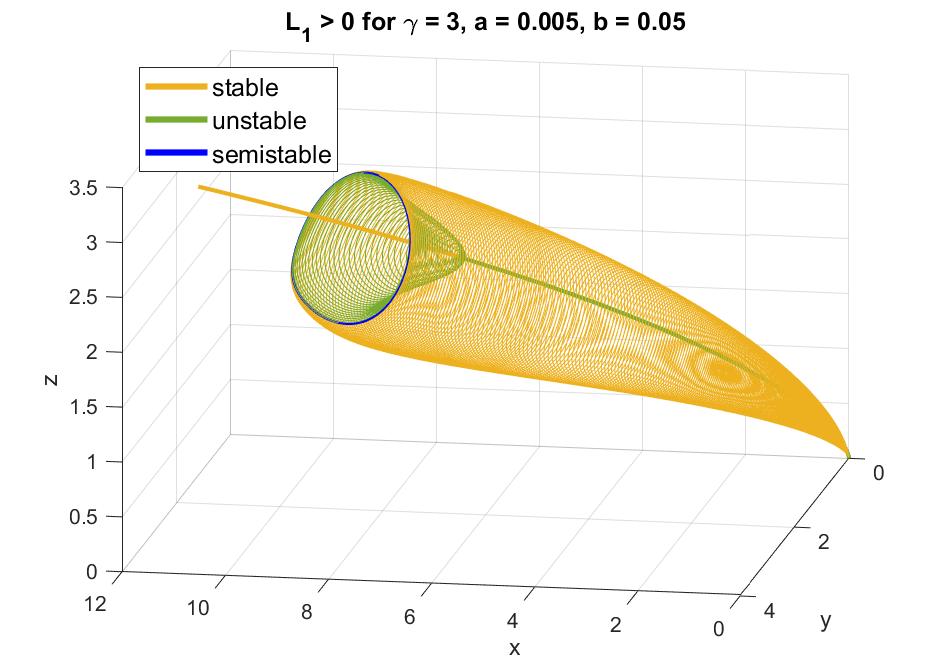}
    \end{tabular}
    \caption{Top row: the sign of the first focal value for the mass-action system \eqref{eq:ode_2LC} (case $\gamma=1$ on the left, case $\gamma=3$ on the right). Bottom row: the unstable limit cycles that are born via a subcritical Andronov--Hopf bifurcation (parameters are taken such that $L_1>0$ holds), surrounded by stable limit cycles (case $\gamma=1$ on the left, case $\gamma=3$ on the right). The stable and unstable limit cycles merge through a fold bifurcation to a semistable limit cycle (shown in blue).}
    \label{fig:2LC_nonhomog}
\end{figure}

%% file: tikz/parallelogram_2LC.tex
\begin{tikzpicture}[scale=1.4]

\node (P1) at (0,1)  {$\mathsf{Y}+\gamma \mathsf{Z}$};
\node (P2) at (1,0)  {$\mathsf{X}+\gamma \mathsf{Z}$};
\node (P3) at (1,2)  {$\mathsf{X}+2\mathsf{Y}$};
\node (P4) at (0,3)  {$3\mathsf{Y}$};

\draw[arrows={-stealth},thick] (P1) to node[below left] {$1$} (P2);
\draw[arrows={-stealth},thick,transform canvas={xshift=3pt}] (P2) to node[right] {$a$} (P3);
\draw[arrows={-stealth},thick,transform canvas={xshift=-3pt}] (P3) to node[left] {$4a$} (P2);
\draw[arrows={-stealth},thick] (P3) to node[above right] {$1$} (P4);
\draw[arrows={-stealth},thick,transform canvas={xshift=-3pt}] (P4) to node[left] {$4b$} (P1);
\draw[arrows={-stealth},thick,transform canvas={xshift=3pt}] (P1) to node[right] {$b$} (P4);

\node at (5.5,1.5) {$\begin{aligned}
\dot{x} &= yz^\gamma - xy^2, \\
\dot{y} &= -yz^\gamma + xy^2 + 2(axz^\gamma - 4axy^2 + byz^\gamma - 4by^3), \\
\dot{z} &= \gamma(-a x z^\gamma + 4a x y^2 - b y z^\gamma + 4 b y^3),
\end{aligned}$};

\end{tikzpicture}

%% file: sections/3p3_2Hopf.tex
\newpage
\subsection{Two Andronov--Hopf points}
\label{subsec:2LC_2Hopf}

Let us take the third planar parallelogram in \Cref{fig:planar_parallelograms} and add a new species, $\mathsf{Z}$, as follows:
\begin{align}\label{eq:ode_2Hopf}
\begin{aligned}
\input{tikz/parallelogram_2Hopf}
\end{aligned}
\end{align}
where we employed special rate constants (for $a,b>0$) and also displayed the associated mass-action differential equation. Similarly to \Cref{subsec:2LC_1Hopf}, this special choice of the rate constants makes the calculations somewhat easier and they still allow us to present some qualitative pictures that were not seen in \Cref{subsec:1LC,subsec:2LC_1Hopf}. Note that $x + 2 y + 4 z$ is a conserved quantity, and therefore the network is mass-conserving. By a short calculation, the set of positive equilibria is the curve
\begin{align*}
(t,t^2/4,t^4/4) \text{ for } t>0.
\end{align*}
Since the stability of a positive equilibrium within its stoichiometric class is determined by the sign of the trace of the Jacobian matrix, we compute the trace along the curve of equilibria and find it is
\begin{align*}
\frac{t^2}{4}[-t^3+(1-4(4a+b))t^2-(4a+b)]  \text{ for } t>0.
\end{align*}
One finds that, as a function of $t$, the trace has exactly two roots if and only if $4a+b<1/16$. Call the two roots $t^{(1)}$ and $t^{(2)}$ (with $0 < t^{(1)} < t^{(2)}$; the dependence on $a$ and $b$ is omitted for readability). To understand the stability of the equilibria at $t^{(1)}$ and $t^{(2)}$, one computes the respective first focal values $L_1^{(1)}$ and $L_1^{(2)}$. Their sign is shown in the left panel of \Cref{fig:prlgrm_wide_bifurcation}, the generic cases are
\begin{itemize}
\item both of $L_1^{(1)}$ and $L_1^{(2)}$ are negative (first row in \Cref{fig:wide}),
\item $L_1^{(1)}$ is negative and $L_1^{(2)}$ is positive (left panel in the second row in \Cref{fig:wide}),
\item both of $L_1^{(1)}$ and $L_1^{(2)}$ are positive (right panel in the second row in \Cref{fig:wide}).
\end{itemize}

\begin{figure}[h]
\centering
\begin{tabular}{cc}
\includegraphics[scale=0.5]{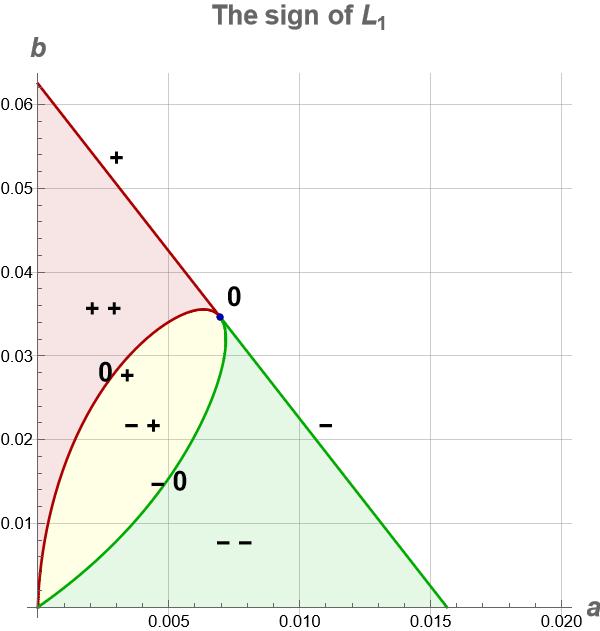} &
\includegraphics[scale=0.5]{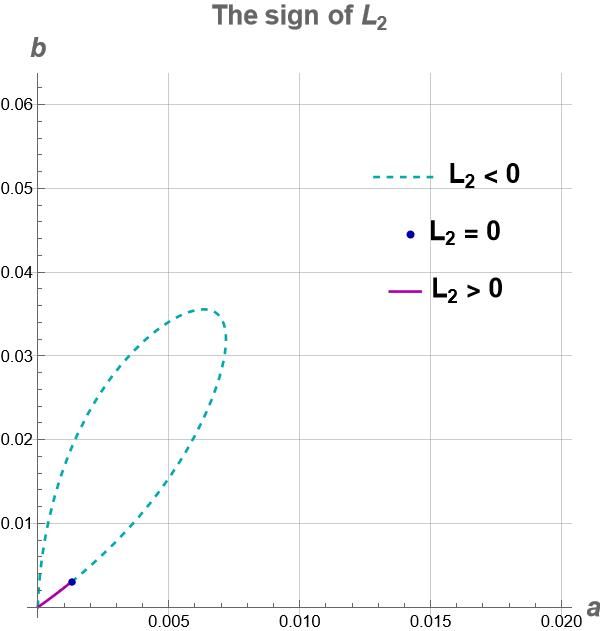}
\end{tabular}
\caption{Left panel: the sign of the first focal values $L_1^{(1)}$ and $L_1^{(2)}$ for $4a+b \leq 1/16$ for the mass-action system \eqref{eq:ode_2Hopf} (since $t^{(1)}$ and $t^{(2)}$ coincide if $4a+b=1/16$, we displayed only one sign there). Right panel: the sign of $L_2$ along the curve $L_1=0$. The third focal value is negative at the unique $(a,b,t)\in\mathbb{R}^3_+$, where $\tr J = L_1 = L_2 = 0$.}
\label{fig:prlgrm_wide_bifurcation}
\end{figure}

\newpage
In the boundary case $4a+b=1/16$, we have $t^{(1)}=t^{(2)}=1/2$. If additionally the first focal value is positive (take for instance $a=1/256$ and $b=12/256$) then the trace, as a function of $t$, vanishes, but does not change sign at $1/2$. The equilibrium at $t=1/2$ is repelling, since the first focal value is positive. Furthermore, since this is the limiting case of $L_1^{(1)}>0$ and $L_1^{(2)}>0$, both for $t$ slightly smaller and slightly larger than $1/2$, the stable equilibrium is surrounded by an unstable limit cycle. Furthermore, by permanence, any unstable equilibrium and any unstable limit cycle is surrounded by a stable limit cycle. Using MATCONT, the picture we get in this case is shown in the left panel in the third row in \Cref{fig:wide}.

Perhaps the most interesting qualitative picture in this paper is obtained by taking parameter values $a$ and $b$ as in the previous paragraph and then perturbing them slightly to make $4a+b>1/16$. Then every positive equilibrium is asymptotically stable, but not all of them are globally stable within their stoichiometric classes, because a torus of stable and unstable limit cycles is created around the curve of positive equilibria as shown in the right panel in the third row in \Cref{fig:wide}.

We conclude this section by explaining how can one construct even three limit cycles. Viewing now $(a,b,t)\in\mathbb{R}^3_+$ as parameter, the trace of the Jacobian matrix vanishes along a surface $\mathcal{M}$. On this surface, there is a curve $\gamma$, where the first focal value vanishes (the projection of this curve to the $(a,b)$-plane is shown in \Cref{fig:prlgrm_wide_bifurcation}). Along this curve, there is a point $(a^*,b^*,t^*)\approx(0.001291,0.003044,0.958228)$, where the second focal value changes sign (this is shown in the right panel in \Cref{fig:prlgrm_wide_bifurcation}). One computes the third focal value at $(a^*,b^*,t^*)$ and finds it is negative. The equilibrium is therefore asymptotically stable. Then do the following steps.
\begin{enumerate}[(i)]
\item First perturb $(a,b,t)$ slightly along the curve $\gamma$ such that the second focal value becomes positive. Then a stable limit cycle $\Gamma_3$ is born via a degenerate supercritical Andronov--Hopf bifurcation and the equilibrium becomes repelling.
\item Then perturb $(a,b,t)$ slightly along the surface $\mathcal{M}$ such that the first focal value becomes negative. Then an unstable limit cycle $\Gamma_2$ is born via a degenerate subcritical Andronov--Hopf bifurcation and the equilibrium becomes asymptotically stable.
\item Finally perturb $(a,b,t)$ slightly away from the surface $\mathcal{M}$ such that the trace of the Jacobian matrix becomes positive. Then a stable limit cycle $\Gamma_1$ is born via a nondegenerate supercritical Andronov--Hopf bifurcation and the equilibrium becomes repelling.
\end{enumerate}

We have thus proved that there exists an $(a,b,t)$ such that the unique positive equilibrium is repelling and is surrounded by three limit cycles (two stable and one unstable one). It would be interesting to continue these limit cycles in MATCONT as $t$ is varied, however, apparently the three limit cycles coexist only in a very tiny region of the parameter space, and thus, it is numerically not easy to create meaningful pictures.

\begin{figure}[!t]
    \centering
    \begin{tabular}{cc}
    \multicolumn{2}{c}{\includegraphics[scale=0.229]{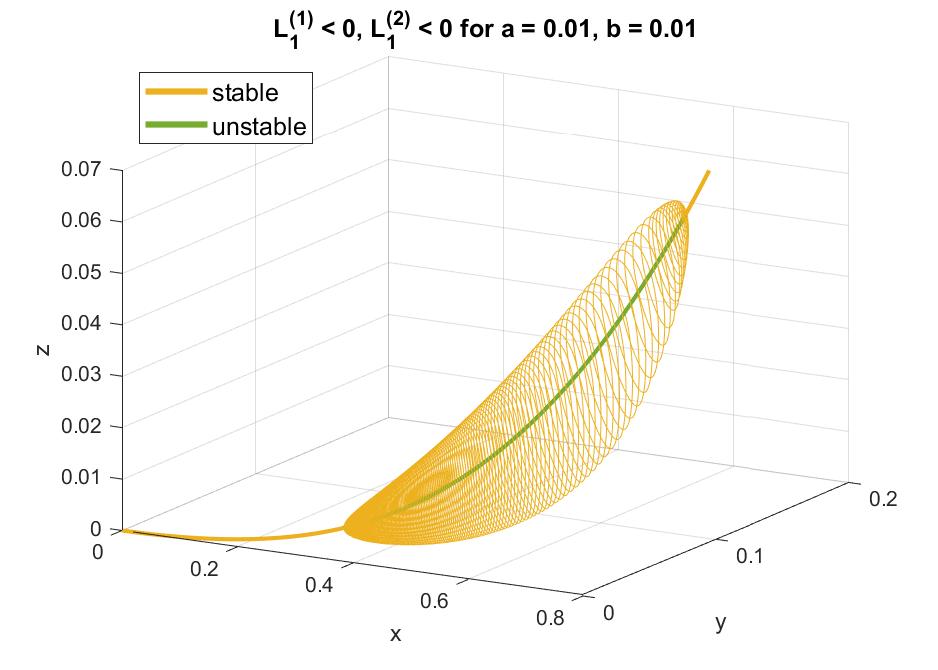}} \\
    \includegraphics[scale=0.229]{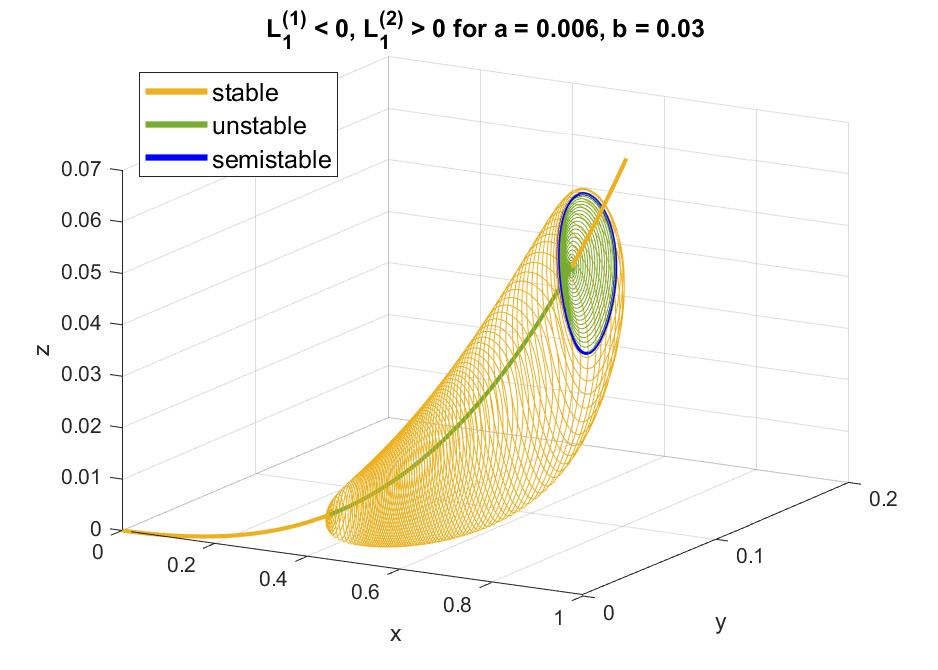} &
    \includegraphics[scale=0.229]{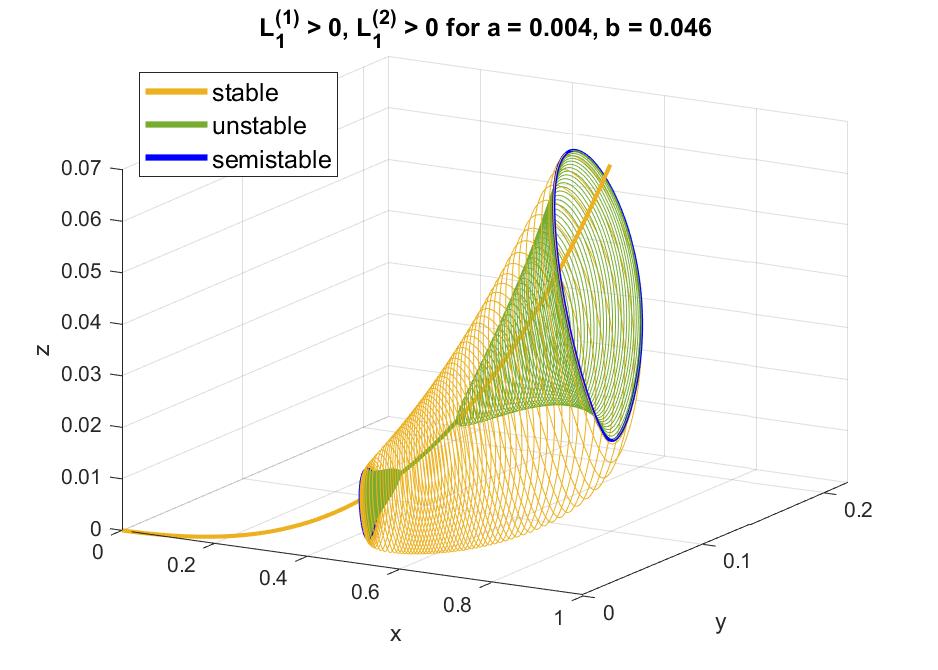} \\
    \includegraphics[scale=0.229]{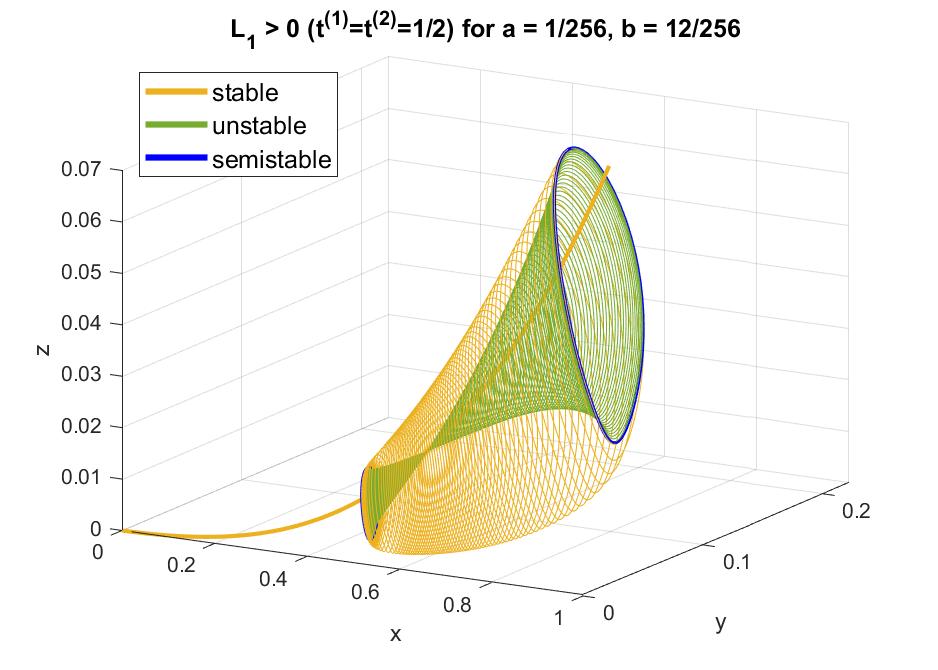} &
    \includegraphics[scale=0.229]{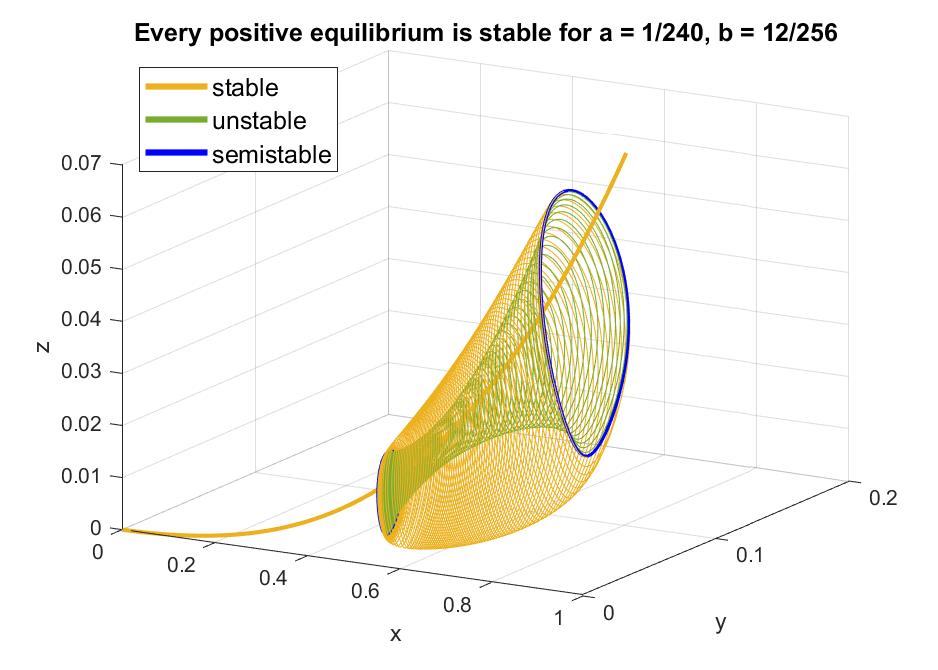}
    \end{tabular}
    \caption{The various shapes that the stable and unstable limit cycles could form for the mass-action system \eqref{eq:ode_2Hopf}. The most interesting is the last one, where all the positive equilibria are asymptotically stable, but not all of them are globally stable, because a torus of limit cycles surrounds the curve of positive equilibria.}
    \label{fig:wide}
\end{figure}

%% file: tikz/parallelogram_2Hopf.tex
\begin{tikzpicture}[scale=1.4]

\node (P1) at (0,1)  {$\mathsf{Y}+\mathsf{Z}$};
\node (P2) at (2,0)  {$2\mathsf{X}+\mathsf{Z}$};
\node (P3) at (2,2)  {$2\mathsf{X}+2\mathsf{Y}$};
\node (P4) at (0,3)  {$3\mathsf{Y}$};

\draw[arrows={-stealth},thick] (P1) to node[below left] {$1$} (P2);
\draw[arrows={-stealth},thick,transform canvas={xshift=3pt}] (P2) to node[right] {$a$} (P3);
\draw[arrows={-stealth},thick,transform canvas={xshift=-3pt}] (P3) to node[left] {$4a$} (P2);
\draw[arrows={-stealth},thick] (P3) to node[above right] {$1$} (P4);
\draw[arrows={-stealth},thick,transform canvas={xshift=-3pt}] (P4) to node[left] {$4b$} (P1);
\draw[arrows={-stealth},thick,transform canvas={xshift=3pt}] (P1) to node[right] {$b$} (P4);

\node at (6.5,1.5) {$\begin{aligned}
\dot{x} &= 2(yz - x^2y^2), \\
\dot{y} &= -yz + x^2y^2 + 2(ax^2z - 4ax^2y^2 + byz - 4by^3), \\
\dot{z} &= -a x^2 z + 4a x^2 y^2 - b y z + 4 b y^3,
\end{aligned}$};

\end{tikzpicture}

%% file: sections/4_bimolecular.tex
\clearpage
\section{Bimolecular networks and limit cycles}
\label{sec:bimolecular}

The \emph{molecularity} of a complex $y \in \mathbb{Z}^n_{\geq0}$ (or $y_1\mathsf{X}_1 + \cdots + y_n\mathsf{X}_n$) is the sum $y_1 + \cdots + y_n \in \mathbb{Z}_{\geq0}$. We say that a reaction network $(V,E)$ (or a mass-action system $(V,E,\kappa)$) is \emph{bimolecular} if the molecularity of each element of $V$ is at most two. None of the examples in \Cref{sec:parallelograms} is bimolecular. In this section we prove that for a rank-two bimolecular reaction network the associated mass-action system does not admit a limit cycle. For $2$ and $3$ species this was proven by P\'ota in \cite{pota:1983} and \cite{pota:1985}, respectively. \Cref{thm:no_limit_cycle_in_2d} below is an extension of these results to arbitrary number of species. On the other hand, there are two famous bimolecular mass-action systems of rank two that do oscillate: the Lotka and the Ivanova networks give rise to centers.

The Lotka network and its associated mass-action system are
\begin{center}
\input{tikz/lotka}
\end{center}
Notice however that the mass-action system
\begin{center}
\input{tikz/lotka_extended}
\end{center}
also gives rise to a center, provided that
\begin{align*}
\sgn(\kappa_1 - \widetilde{\kappa}_1+\lambda+\nu) = \sgn(\kappa_1 - \widetilde{\kappa}_1-\lambda-\mu) = \sgn(\kappa_2 - \widetilde{\kappa}_2) = \sgn(\kappa_3 - \widetilde{\kappa}_3) \neq 0.
\end{align*}
In the above, we allow some of $\kappa_1$, $\widetilde{\kappa}_1$, $\kappa_2$, $\widetilde{\kappa}_2$, $\kappa_3$, $\widetilde{\kappa}_3$, $\lambda$, $\mu$, $\nu$ to vanish, in which case the corresponding reaction is not present. Finally, we note that the sum of the quadratic terms in $\dot{x}$ and $\dot{y}$ is nonpositive. This is because for a bimolecular reaction network in any reaction the molecularity of the  product complex cannot be larger than the  molecularity of the reactant complex if the latter equals two.

The Ivanova network and its associated mass-action system are
\begin{center}
\input{tikz/ivanova}
\end{center}
Notice however that the mass-action system
\begin{center}
\input{tikz/ivanova_extended}
\end{center}
also gives rise to a center, provided that
\begin{align*}
\sgn(\kappa_1 - \widetilde{\kappa}_1) = \sgn(\kappa_2 - \widetilde{\kappa}_2) = \sgn(\kappa_3 - \widetilde{\kappa}_3) \neq 0.
\end{align*}
In the above, we allow some of $\kappa_1$, $\widetilde{\kappa}_1$, $\kappa_2$, $\widetilde{\kappa}_2$, $\kappa_3$, $\widetilde{\kappa}_3$ to vanish, in which case the corresponding reaction is not present. Interestingly, the Ivanova network can be obtained from the Lotka network by adding a new species, $\mathsf{Z}$, in a way that the molecularity of every complex becomes two.

Observe that there is another way for a rank-two bimolecular mass-action system to admit periodic solutions:
\begin{center}
\input{tikz/lotka_dummy}
\end{center}
Here, the positive stoichiometric classes are given by $z=c$ for some $c>0$. For $c<\frac{\kappa_2}{\kappa_4}$, the unique positive equilibrium is a global center (the dynamics is the same as for the Lotka network), while for $c \geq \frac{\kappa_2}{\kappa_4}$ there is no positive equilibrium. 

The following theorem says that essentially the Lotka and the Ivanova differential equations are the only ones with a periodic solution that are derived from a rank-two bimolecular reaction network. 
\begin{theorem} \label{thm:no_limit_cycle_in_2d}
Suppose that a rank-two bimolecular mass-action system has a periodic solution in a positive stoichiometric class $\mathcal{P}$. Then the dynamics in $\mathcal{P}$ is described either by
\begin{align*}
\dot{x} &= x(a-by),\\
\dot{y} &= y(b'x-c)
\end{align*}
for some $a,b,b',c\in\mathbb{R}$ with $\sgn a = \sgn b = \sgn b' = \sgn c \neq 0$ and $b'\leq b$ or by
\begin{align*}
\dot{x} &= x(az-by),\\
\dot{y} &= y(bx-cz),\\
\dot{z} &= z(cy-ax)
\end{align*}
for some $a,b,c\in\mathbb{R}$ with $\sgn a = \sgn b = \sgn c \neq 0$. In particular, there is a unique positive equilibrium in $\mathcal{P}$, every non-equilibrium solution is periodic, and there is no limit cycle.
\end{theorem}
\begin{proof}
For $k=1,\ldots,n$ we collect the terms in $\dot{x}_k$ according to their degree in $x_k$. To ease the notation, let $x_{-k}=(x_1, \dots, x_{k-1}, x_{k+1}, \dots , x_n) \in {\mathbb R}^{n-1}$ for $x \in \mathbb{R}^n$. With this, 
\begin{align*}
\dot{x}_k = a_k(x_{-k}) + x_k b_k(x_{-k}) + c_k x_k^2 \text{ for } k = 1,\ldots,n,
\end{align*}
where the function $a_k$ collects all the terms without $x_k$, the middle term collects all the terms, where $x_k$ appears linearly, and finally, the only term which is quadratic in $x_k$ is denoted by $c_k x_k^2$. Note that $a_k(x_{-k})\geq0$, because in a reaction with $\mathsf{X}_k$ not being a reactant species, $\mathsf{X}_k$ can only be gained. Also, $c_k\leq0$, since $\mathsf{X}_k$ can only be consumed in a bimolecular reaction whose reactant complex is $2\mathsf{X}_k$.

After dividing the vectorfield by $\prod_{i=1}^nx_i$, the divergence equals to
\begin{align*}
\frac{1}{\prod_{i=1}^nx_i}\sum_{k=1}^n\left(-\frac{a_k(x_{-k})}{x_k}+c_k x_k\right),
\end{align*}
which is nonpositive in $\mathbb{R}^n_+$ by the above discussion. By a multidimensional version of the Bendixson--Dulac test, the existence of a periodic solution together with $\dim \mathcal{P} = 2$ implies that the divergence vanishes everywhere, see e.g.\ \cite[Satz 1]{schneider:1969} for the $n=3$ case and \cite[Theorem 3.3, Remark (3)]{li:1995} for the general case. Therefore, both $a_k(x_{-k})$ and $c_k$ vanish for all $k = 1, \ldots, n$ and
\begin{align} \label{eq:LV}
\dot{x}_k = x_k b_k(x_{-k}) = x_k \left(r_k + \sum_{i\neq k}b_{ki}x_i\right) \text{ for } k = 1, \ldots, n.
\end{align}

If there exists a $k'$ such that $r_{k'} = 0$ and $b_{k'i}=0$ for all $i\neq k'$ then $\dot{x}_{k'} = 0$ and thus $x_{k'}\equiv d$ for some $d>0$. In this case, for each $k \neq k'$ we update $r_k$ to be $r_k + b_{kk'}d$ and omit the equation for $\dot{x}_{k'}$. For the rest of this proof, we assume that the differential equation \eqref{eq:LV} is such that
\begin{align} \label{eq:no_dummy}
\text{for all $k = 1, \ldots, n$ at least one of the $n$ numbers $r_k$ and $b_{ki}$ ($i\neq k$) is nonzero.}
\end{align}

\textbf{Case $n=2$.} The differential equation \eqref{eq:LV} takes the form
\begin{align*}
    \dot{x}_1 &= x_1(r_1 + b_{12}x_2),\\
    \dot{x}_2 &= x_2(r_2 + b_{21}x_1).
\end{align*}
Since there exists a periodic solution, $\sgn r_1 = -\sgn b_{12} \neq 0$ and $\sgn r_2 = -\sgn b_{21} \neq 0$ follow. Furthermore, $\sgn b_{12} = -\sgn b_{21}$ (otherwise the unique positive equilibrium inside the closed orbit would be a saddle with index $-1$, a contradiction). Finally, $b_{12}+b_{21}\leq 0$ follows from the assumption that every product complex has molecularity at most two.

\textbf{Case $n=3$.} The differential equation \eqref{eq:LV} takes the form
\begin{align}\label{eq:LV_n3}
\begin{split}
\dot{x}_1 &= x_1(r_1 + b_{12}x_2 + b_{13}x_3),\\
\dot{x}_2 &= x_2(r_2 + b_{21}x_1 + b_{23}x_3),\\
\dot{x}_3 &= x_3(r_3 + b_{31}x_1 + b_{32}x_2).
\end{split}
\end{align}
Since the stoichiometric subspace is two-dimensional, there exists a $d\in\mathbb{R}^3\setminus\{0\}$ such that $d_1 \dot{x}_1 + d_2 \dot{x}_2 + d_3 \dot{x}_3 = 0$. If exactly one coordinate of $d$ is nonzero, say $d_1$, then $\dot{x}_1 = 0$, contradicting the assumption \eqref{eq:no_dummy}. If exactly two coordinates of $d$ are nonzero, say $d_1$ and $d_2$, then $r_1 = r_2 = b_{13} = b_{23} = 0$ follows. Taking also into account \eqref{eq:no_dummy}, $b_{12}\neq 0$ and $b_{21} \neq 0$ hold. But then both $x_1$ and $x_2$ are strictly monotonic over time, contradicting the existence of a periodic solution. Thus, each of $d_1, d_2, d_3$ is nonzero. Hence, $d_1 \dot{x}_1 + d_2 \dot{x}_2 + d_3 \dot{x}_3 = 0$ implies that $r_1 = r_2 = r_3 = 0$. Furthermore, the existence of a periodic solution and the assumption \eqref{eq:no_dummy} together imply $\sgn b_{12} = -\sgn b_{13} \neq 0$, $\sgn b_{21} = -\sgn b_{23} \neq 0$, and $\sgn b_{31} = -\sgn b_{32} \neq 0$. Since for reactions with reactant complexes $\mathsf{X}_1 + \mathsf{X}_2$, $\mathsf{X}_1 + \mathsf{X}_3$, or $\mathsf{X}_2 + \mathsf{X}_3$, the molecularity cannot increase (by the assumption that there is no product complex with molecularity higher than two), it follows that $\dot{x}_1 + \dot{x}_2 + \dot{x}_3 \leq 0$. Thus, along a periodic solution, $\dot{x}_1 + \dot{x}_2 + \dot{x}_3 = 0$, implying $b_{12} = -b_{21}$, $b_{13} = -b_{31}$, and $b_{23} = -b_{32}$.

\textbf{Case $n \geq 4$.} We now show that this case is actually empty. Since the stoichiometric subspace is $2$-dimensional, the linear conservation laws form an $(n-2)$-dimensional subspace of $\mathbb{R}^n$. In such a subspace there always exists a nonzero vector $d$ with at most three nonzero entries. Assume that $d$ has a support that is minimal w.r.t.\ inclusion. If $d$ has exactly one or two nonzero entries then we arrive at a contradiction in the same way as in the $n=3$ case above. If $d$ has exactly three nonzero entries, say $d_1$, $d_2$, $d_3$, then $(x_1,x_2,x_3)$ evolves according to the equations \eqref{eq:LV_n3}, because $b_{1i} = b_{2i} = b_{3i} = 0$ for all $i = 4, \ldots, n$ follows. By the minimality of the support of $d$, these three variables already occupy two dimensions, and thus $\dot{x}_4 = \cdots = \dot{x}_n = 0$, contradicting the assumption \eqref{eq:no_dummy}.
\end{proof}

\Cref{thm:no_limit_cycle_in_2d} shows that the rank of a bimolecular mass-action system with a limit cycle is at least three. We analyse some simple bimolecular oscillators of rank three in \cite{boros:hofbauer:2022b}.
%\begin{corollary}
%The rank of a bimolecular mass-action system with a limit cycle is at least three.
%\end{corollary}

We conclude with a remark. For bimolecular networks we required that every complex's molecularity is at most two. If one relaxes this and imposes only that every reactant complex's molecularity is at most two, while it is allowed to have a product complex with molecularity three then limit cycles in rank-two mass-action systems are not excluded in general. Indeed, the mass-action system
\begin{center}
\input{tikz/kamenetsky}
\end{center}
by Frank-Kamenetsky and Salnikov \cite{frank-kamenetsky:salnikov:1943} admits a stable limit cycle.

%% file: tikz/lotka.tex
\begin{tikzpicture}[scale=1.4]

\node (P1) at (0,0)    {$\mathsf{X}+\mathsf{Y}$};
\node (P2) at (1.5,0)    {$2\mathsf{Y}$};
\node (P3) at (0,-2/3) {$\mathsf{X}$};
\node (P4) at (1.5,-2/3) {$2\mathsf{X}$};
\node (P5) at (0,-4/3)   {$\mathsf{Y}$};
\node (P6) at (1.5,-4/3)   {$\mathsf{0}$};

\draw[arrows={-stealth},thick] (P1) to node[above] {$\kappa_1$} (P2);
\draw[arrows={-stealth},thick] (P3) to node[above] {$\kappa_2$} (P4);
\draw[arrows={-stealth},thick] (P5) to node[above] {$\kappa_3$} (P6);

\node at (3.5,-2/3) {$\begin{aligned}
\dot{x} &= \kappa_2 x - \kappa_1 x y, \\
\dot{y} &= \kappa_1 x y - \kappa_3 y.
\end{aligned}$};

\end{tikzpicture}

%% file: tikz/lotka_extended.tex
\begin{tikzpicture}[scale=1.6]

\node (PXY) at (1,1) {$\mathsf{X}+\mathsf{Y}$};
\node (P2Y) at (0,2) {$2\mathsf{Y}$};
\node (PX)  at (1,0) {$\mathsf{X}$};
\node (P2X) at (2,0) {$2\mathsf{X}$};
\node (PY)  at (0,1) {$\mathsf{Y}$};
\node (P0)  at (0,0) {$\mathsf{0}$};

\draw[arrows={-stealth},thick] (PXY) to node[above right] {$\kappa_1$} (P2Y);
\draw[arrows={-stealth},thick] (PXY) to node[above right] {$\widetilde{\kappa}_1$} (P2X);
\draw[arrows={-stealth},thick] (PXY) to node[above] {$\lambda$} (P0);
\draw[arrows={-stealth},thick] (PXY) to node[right] {$\mu$} (PX);
\draw[arrows={-stealth},thick] (PXY) to node[above] {$\nu$} (PY);
\draw[arrows={-stealth},thick] (PX) to node[below] {$\kappa_2$} (P2X);
\draw[arrows={-stealth},thick] (PX) to node[below] {$\widetilde{\kappa}_2$} (P0);
\draw[arrows={-stealth},thick] (PY) to node[left] {$\kappa_3$} (P0);
\draw[arrows={-stealth},thick] (PY) to node[left] {$\widetilde{\kappa}_3$} (P2Y);

\node at (5,1) {$\begin{aligned}
\dot{x} &= (\kappa_2 - \widetilde{\kappa}_2)x - (\kappa_1 -\widetilde{\kappa}_1 + \lambda + \nu) x y, \\
\dot{y} &= (\kappa_1 -\widetilde{\kappa}_1 - \lambda - \mu) x y - (\kappa_3 - \widetilde{\kappa}_3) y
\end{aligned}$};

\end{tikzpicture}

%% file: tikz/ivanova.tex
\begin{tikzpicture}[scale=1.4]

\node (P1) at (0,0)    {$\mathsf{X}+\mathsf{Y}$};
\node (P2) at (1.5,0)    {$2\mathsf{Y}$};
\node (P3) at (0,-2/3) {$\mathsf{X}+\mathsf{Z}$};
\node (P4) at (1.5,-2/3) {$2\mathsf{X}$};
\node (P5) at (0,-4/3)   {$\mathsf{Y}+\mathsf{Z}$};
\node (P6) at (1.5,-4/3)   {$2\mathsf{Z}$};

\draw[arrows={-stealth},thick] (P1) to node[above] {$\kappa_1$} (P2);
\draw[arrows={-stealth},thick] (P3) to node[above] {$\kappa_2$} (P4);
\draw[arrows={-stealth},thick] (P5) to node[above] {$\kappa_3$} (P6);

\node at (3.5,-2/3) {$\begin{aligned}
\dot{x} &= \kappa_2 x z - \kappa_1 x y, \\
\dot{y} &= \kappa_1 x y - \kappa_3 y z, \\
\dot{z} &= \kappa_3 y z - \kappa_2 x z.
\end{aligned}$};

\end{tikzpicture}

%% file: tikz/ivanova_extended.tex
\begin{tikzpicture}[scale=1.6]

\node (PXY) at (1,1) {$\mathsf{X}+\mathsf{Y}$};
\node (P2Y) at (0,2) {$2\mathsf{Y}$};
\node (PXZ)  at (1,0) {$\mathsf{X}+\mathsf{Z}$};
\node (P2X) at (2,0) {$2\mathsf{X}$};
\node (PYZ)  at (0,1) {$\mathsf{Y}+\mathsf{Z}$};
\node (P2Z)  at (0,0) {$2\mathsf{Z}$};

\draw[arrows={-stealth},thick] (PXY) to node[above right] {$\kappa_1$} (P2Y);
\draw[arrows={-stealth},thick] (PXY) to node[above right] {$\widetilde{\kappa}_1$} (P2X);
\draw[arrows={-stealth},thick] (PXZ) to node[below] {$\kappa_2$} (P2X);
\draw[arrows={-stealth},thick] (PXZ) to node[below] {$\widetilde{\kappa}_2$} (P2Z);
\draw[arrows={-stealth},thick] (PYZ) to node[left] {$\kappa_3$} (P0);
\draw[arrows={-stealth},thick] (PYZ) to node[left] {$\widetilde{\kappa}_3$} (P2Y);

\node at (4.5,1) {$\begin{aligned}
\dot{x} &= (\kappa_2 - \widetilde{\kappa}_2) x z - (\kappa_1 -\widetilde{\kappa}_1) x y, \\
\dot{y} &= (\kappa_1 -\widetilde{\kappa}_1) x y - (\kappa_3 - \widetilde{\kappa}_3) y z, \\
\dot{z} &= (\kappa_3 -\widetilde{\kappa}_3) y z - (\kappa_2 - \widetilde{\kappa}_2) x z
\end{aligned}$};

\end{tikzpicture}

%% file: tikz/lotka_dummy.tex
\begin{tikzpicture}[scale=1.4]

\node (P1) at (0,0)    {$\mathsf{X}+\mathsf{Y}$};
\node (P2) at (1.5,0)    {$2\mathsf{Y}$};
\node (P3) at (0,-2/3) {$\mathsf{X}$};
\node (P4) at (1.5,-2/3) {$2\mathsf{X}$};
\node (P5) at (0,-4/3)   {$\mathsf{Y}$};
\node (P6) at (1.5,-4/3)   {$\mathsf{0}$};
\node (P7) at (0,-2)   {$\mathsf{X}+\mathsf{Z}$};
\node (P8) at (1.5,-2)   {$\mathsf{Z}$};

\draw[arrows={-stealth},thick] (P1) to node[above] {$\kappa_1$} (P2);
\draw[arrows={-stealth},thick] (P3) to node[above] {$\kappa_2$} (P4);
\draw[arrows={-stealth},thick] (P5) to node[above] {$\kappa_3$} (P6);
\draw[arrows={-stealth},thick] (P7) to node[above] {$\kappa_4$} (P8);

\node at (4.5,-1) {$\begin{aligned}
\dot{x} &= (\kappa_2 - \kappa_4 z) x - \kappa_1 x y, \\
\dot{y} &= \kappa_1 x y - \kappa_3 y, \\
\dot{z} &= 0.
\end{aligned}$};

\end{tikzpicture}

%% file: tikz/kamenetsky.tex
\begin{tikzpicture}[scale=1.4]

\node (P1) at (0,0)    {$\mathsf{X}+\mathsf{Y}$};
\node (P2) at (1.5,0)    {$2\mathsf{Y}$};
\node (P3) at (-0.75,-2/3) {$\mathsf{X}$};
\node (P4) at (0.75,-2/3) {$2\mathsf{X}$};
\node (P5) at (2.25,-2/3) {$3\mathsf{X}$};
\node (P6) at (0,-4/3)   {$\mathsf{Y}$};
\node (P7) at (1.5,-4/3)   {$\mathsf{0}$};

\draw[arrows={-stealth},thick] (P1) to node[above] {$\kappa_1$} (P2);
\draw[arrows={-stealth},thick] (P3) to node[above] {$\kappa_2$} (P4);
\draw[arrows={-stealth},thick] (P4) to node[above] {$\kappa_4$} (P5);
\draw[arrows={-stealth},thick,transform canvas={yshift=2pt}] (P6) to node[above] {$\kappa_3$} (P7);
\draw[arrows={-stealth},thick,transform canvas={yshift=-2pt}] (P7) to node[below] {$\kappa_5$} (P6);

\node at (4.5,-2/3) {$\begin{aligned}
\dot{x} &= \kappa_2 x - \kappa_1 x y + \kappa_4 x^2, \\
\dot{y} &= \kappa_1 x y - \kappa_3 y + \kappa_5
\end{aligned}$};

\end{tikzpicture}